\documentclass[final,leqno]{siamltex704}
\usepackage{amsmath,amsfonts,amssymb}
\usepackage{graphicx}
\usepackage[notcite,notref]{showkeys}
\usepackage{mathrsfs}
\usepackage{appendix}
\usepackage{color}
\usepackage{lineno}
\usepackage{float}
\usepackage{dsfont}
\usepackage{pifont}
\usepackage{wrapfig} 
\usepackage{hyperref}
\usepackage{multirow}
\usepackage{bm} 
\numberwithin{equation}{section}
\def\3bar{{|\hspace{-.02in}|\hspace{-.02in}|}}
\def\E{{\mathcal{E}}}
\def\T{{\mathcal{T}}}

\def\pT{{\partial T}}

\def\W{{\mathcal{W}}}
\def\ljump{{[\![}}
\def\rjump{{]\!]}}

\def\bn{{\bm{n}}}

\newtheorem{algorithm}{Algorithm}[section]

 \title {A Parallel iterative Algorithm  for  primal-dual  weak Galerkin Schemes}
 
\author{
Chunmei Wang \thanks{Department of Mathematics, University of Florida, Gainesville, FL 32611, USA (chunmei.wang@ufl.edu). The research of Chunmei Wang was partially supported by National Science Foundation Award DMS-2136380.}
 \and
 Junping Wang\thanks{Division of Mathematical
 Sciences, National Science Foundation, Alexandria, VA 22314
 (jwang@nsf.gov). The research of Junping Wang was supported in part by the
 NSF IR/D program, while working at National Science Foundation.
 However, any opinion, finding, and conclusions or recommendations
 expressed in this material are those of the author and do not
 necessarily reflect the views of the National Science Foundation.}
 }

\begin{document}

\maketitle

\begin{abstract}
This paper presents and analyzes a parallelizable iterative procedure based on domain decomposition for primal-dual weak Galerkin (PDWG) finite element methods applied to the Poisson equation. The existence and uniqueness of the PDWG solution are established. Optimal order of error estimates are derived in both a discrete norm and the $L^2$ norm. The convergence analysis is conducted for domain decompositions into individual elements associated with the PDWG methods, which can be extended to larger subdomains without any difficulty.
\end{abstract}

\begin{keywords}
primal-dual, weak Galerkin, finite element methods, Poisson equation, parallelizable iterative methods, domain decomposition.
\end{keywords}

\begin{AMS}
Primary, 65N30, 65N15, 65N12, 74N20; Secondary, 35B45, 35J50,
35J35
\end{AMS}

\pagestyle{myheadings}
\section{Introduction}
In this paper we consider numerical methods for the Poisson equation with boundary conditions. The model problem seeks an unknown function $u$ satisfying
\begin{equation}\label{model}
\begin{split}
\Delta u=&f, \qquad
\text{in}\quad \Omega,\\ u=&g, \qquad \text{on}\quad \Gamma,\\ 
\end{split}
\end{equation}
where $\Omega\subset \mathbb R^d (d=2, 3)$ is an open bounded and connected  domain with Lipschitz continuous boundary  $\partial \Omega$, denoted by  $\Gamma$. 

The Weak Galerkin (WG) finite element method is a recently developed numerical technique for solving partial differential equations (PDEs). In this method, differential operators in the variational formulation are reconstructed or approximated within a framework that emulates the theory of distributions for piecewise polynomials. The usual regularity requirements for the approximating functions are offset by carefully designed stabilizers. The WG method has been studied for numerous model PDEs, as evidenced by an incomplete list of references \cite{wg1, wg2, wg3, wg4, wg5, wg6, wg7, wg8, wg9, wg10, wg11, wg12, wg13, wg14, wg15, wg16,
wg17, wg18, wg19, wg20, wg21, itera, wy3655}. These studies demonstrate the WG method's potential as a practically useful tool in scientific computing.
What sets WG methods apart from other finite element methods is the use of weak derivatives and weak continuities in designing numerical schemes based on the conventional weak forms of the underlying PDE problems. This inherent structural flexibility makes WG methods particularly suitable for a wide range of PDEs, ensuring stability and accuracy in their approximations.

A notable advancement in the WG method is the development of the "Primal-Dual Weak Galerkin (PDWG)" approach. This method addresses problems that are challenging for traditional numerical techniques \cite{pdwg1, pdwg2, pdwg3, pdwg4, pdwg5, pdwg6, pdwg7, pdwg8, pdwg9, pdwg10, pdwg11, pdwg12, pdwg13, pdwg14, pdwg15}. The core concept of PDWG is to view the numerical solutions as a constrained minimization of functionals, with constraints mimicking the weak formulation of the PDEs using weak derivatives. This approach yields an Euler-Lagrange equation that provides a symmetric scheme incorporating both the primal variable and the dual variable (Lagrange multiplier).

Given the large size of the computational problem, it is essential to design efficient and parallelizable iterative algorithms for the PDWG scheme. Iterative algorithms have been developed for WG methods using domain decomposition techniques \cite{tw, itera, tw2017, qzw, lxz}. Our approach is inspired by Despres' work on a Helmholtz problem \cite{9} and a Helmholtz-like problem related to Maxwell's equations \cite{10, 11}. It is important to note that the convergence in \cite{9, 10, 11} was established for the differential problems in their strong form, with numerical results supporting the iterative procedures for the discrete case.
Additionally, Douglas et al. \cite{j} introduced a parallel iterative procedure for second-order partial differential equations approximated by mixed finite element methods. Recently, Wang et al. \cite{itera} proposed a similar parallel iterative procedure for the weak Galerkin method for second-order elliptic problems. The goal of this paper is to extend the results of Douglas et al. and Wang et al. \cite{itera} to the PDWG finite element methods. Specifically, the iterative procedure developed in this paper for the PDWG method can be naturally and easily implemented on a massively parallel computer by assigning each subdomain to its own processor.

The paper is structured as follows: Section 2 provides a brief review of the weak formulation and discrete weak differential operators. Section 3 details the primal-dual WG method for solving the Poisson equation \eqref{model} based on the weak form \eqref{weakform}. Section 4 establishes the existence and uniqueness of the primal-dual WG scheme proposed in Section 3. In Section 5, we derive the error equations for the primal-dual WG scheme. Section 6 presents the optimal order error estimates in discrete norms, while Section 7 focuses on establishing  an optimal order error estimate in the $L^2$ norm. Section 8 introduces domain decompositions for the primal-dual WG method, followed by a description of a parallel iterative procedure for the PDWG finite element method in Section 9. Finally, Section 10 provides a convergence analysis for the parallel iterative scheme.

Throughout the paper, we use the standard notations for Sobolev spaces  
and norms. For any open bounded domain $D\subset \mathbb{R}^d$ with
Lipschitz continuous boundary, denote by $\|\cdot\|_{s, D}$, $|\cdot|_{s, D}$ and $(\cdot,\cdot)_{s, D}$ the norm, seminorm and the inner product in the Sobolev space $H^s(D)$ for $s\ge 0$, respectively. 
The space $H^0(D)$ coincides with $L^2(D)$, where the norm and the inner product are denoted by $\|\cdot \|_{D}$ and $(\cdot,\cdot)_{D}$, respectively. When $D=\Omega$, or when the domain of integration is clear from the context,  the subscript $D$ is dropped in the norm and the inner product notation.

\section{Weak Formulations and Discrete Weak Differential Operators} This section will introduce the weak formulation of the Poisson equation  \eqref{model} and briefly review the discrete weak differential operator. 

The weak formulation of the Poisson equation 
\eqref{model} seeks $u\in L^2(\Omega)$ satisfying
\begin{equation}\label{weakform}
 (u,  \Delta \sigma) = (f, \sigma)+\langle g,  \nabla \sigma \cdot \bn\rangle_{\Gamma}, \quad \forall \sigma\in H^2(T).
\end{equation}


Let ${\cal T}_h$ be a partition of the domain $\Omega$ into polygons in 2D or polyhedra in 3D which is shape regular in the sense of \cite{wy3655}. Denote by ${\mathcal E}_h$ the set of all edges or flat faces in ${\cal T}_h$ and  ${\mathcal E}_h^0={\mathcal E}_h \setminus \partial\Omega$ the set of all interior edges or flat faces. Denote by $h_T$ the meshsize of $T\in {\cal T}_h$ and
$h=\max_{T\in {\cal T}_h}h_T$ the meshsize for the partition ${\cal T}_h$. 

Let $T\in \mathcal{T}_h$ be a polygonal or polyhedral region with
boundary $\partial T$. For any $\phi\in H^{1}(T)$ and any polynomial $\psi$, the following trace inequalities hold true  \cite{wy3655}
\begin{equation}\label{tracein}
 \|\phi\|^2_{\partial T} \lesssim h_T^{-1}\|\phi\|_T^2+h_T  \|\phi\|_{1,T}^2, \qquad
 \|\psi\|^2_{\partial T} \lesssim h_T^{-1}\|\psi\|_T^2.
\end{equation}

A weak function on $T\in\T_h$ is denoted by a triplet $\sigma=\{\sigma_0,\sigma_b, \sigma_n\}$ such that $\sigma_0\in L^2(T)$, $\sigma_b\in L^{2}(\partial T)$ and $\sigma_n\in L^{2}(\partial T)$. The first and the second components, namely $\sigma_0$ and $\sigma_b$, represent the values of $\sigma$ in the interior and on the boundary of $T$ respectively. The third component $\sigma_n$ can be understood as the value of $\nabla \sigma \cdot \bn$ on $\pT$, where $\bn$ is an unit outward normal vector on $\pT$. 
Note that $\sigma_b$ and $\sigma_n$ may not necessarily be the traces of $\sigma_0$ and $\nabla \sigma_0 \cdot \bn$ on $\partial T$. Denote by $\W(T)$ the space of all weak functions on $T$; i.e.,
\begin{equation}\label{2.1}
\W(T)=\{\sigma=\{\sigma_0,\sigma_b, \sigma_n \}: \sigma_0\in L^2(T), \sigma_b\in
L^{2}(\partial T), \sigma_n\in L^{2}(\partial T)\}.
\end{equation}

 The weak Laplacian operator of $\sigma\in \W(T)$, denoted by $\Delta_w  \sigma$, is defined as a linear functional such that
 \begin{eqnarray*}
(\Delta_w \sigma, \phi)_T & :=& (\sigma_0, \Delta \phi)_T -\langle \sigma_b, \nabla \phi\cdot \bm{n}\rangle_{\partial T} +
\langle \sigma_n,\phi\rangle_{\partial T},
\end{eqnarray*}
for all $\phi\in H^2(T)$. 

Denote by $P_r(T)$ the space of polynomials on the element $T$ with degree no more than $r$. A discrete version of $\Delta_w  \sigma$, denoted by $\Delta_{w, r, T}\sigma$, is defined as the unique polynomial in $P_r(T)$ satisfying
\begin{eqnarray}\label{disvergence}
 (\Delta_w \sigma, w)_T & :=& (\sigma_0, \Delta w)_T -\langle \sigma_b, \nabla w\cdot \bm{n}\rangle_{\partial T} +
\langle \sigma_n, w\rangle_{\partial T}, \qquad \forall w\in P_r(T).
\end{eqnarray} 

\section{Primal-Dual Weak Galerkin Scheme}\label{Section:WGFEM}
Let $W_k(T)$ be the local discrete weak function space; i.e.,
$$
W_k(T)=\{\{\sigma_0, \sigma_b, \sigma_n\}: \sigma_0\in P_k(T), \sigma_b\in P_{k-1}(e),  \sigma_n \in P_{k-1}(e),e\subset \partial T\}.
$$
Patching $W_k(T)$ over all the elements $T\in {\cal T}_h$
through a common value of $\sigma_b$ and $\sigma_n$ on the interior interface $\E_h^0$, we obtain a global weak finite element space $W_h$; i.e.,
$$
W_h=\big\{\{\sigma_0,\sigma_b, \sigma_n\}:\{\sigma_0,\sigma_b, \sigma_n\}|_T\in W_k(T), \forall T\in \mathcal{T}_h \big\}.
$$
We further introduce the subspace of $W_h$ with homogeneous Dirichlet boundary value, denoted by $W_h^0$; i.e.,
$$
W_h^0=\{v\in W_h: \sigma_b=0 \ \text{on}\ \Gamma\}.
$$
Let $M_h$ be the finite element space consisting of piecewise polynomials of degree $k-1$; i.e.,
$$
M_h=\{w: w|_T\in P_{k-1}(T),  \forall T\in \mathcal{T}_h\}.
$$

For simplicity,  for any $\sigma=\{\sigma_0, \sigma_b, \sigma_n\}\in
W_h$, denote by $\Delta_w \sigma$ the discrete weak Laplacian operator
$\Delta_{w, k-1, T} \sigma$ computed by using
(\ref{disvergence}) on each element $T$; i.e.,
$$
(\Delta_w \sigma)|_T=\Delta_{w,k-1,T}(\sigma|_T), \qquad \forall \sigma\in W_h.
$$
 
On each edge or face $e\subset\partial T$, denote by $Q_b$  the $L^2$ projection operator onto $P_{k-1}(e)$. For any $ \lambda, \sigma \in W_h$, and $u\in M_h$, we introduce the following bilinear forms
\begin{eqnarray*}
s(\lambda, \sigma)&=&\sum_{T\in \mathcal{T}_h} s_T(\lambda, \sigma),\label{EQ:s-form}\\
b(u, \sigma)&=&\sum_{T\in \mathcal{T}_h}(u, \Delta_w \sigma)_T,\label{EQ:b-form}
\end{eqnarray*}
where
 \begin{equation*}\label{EQ:sT-form}
\begin{split}
s_T(\lambda, \sigma)=&h_T^{-3}\langle Q_b\lambda_0-\lambda_b, Q_b\sigma_0-\sigma_b\rangle_{\partial T}
 +h_T^{-1}\langle \nabla \lambda_0 \cdot \bn-\lambda_n, \nabla \sigma_0 \cdot \bn-\sigma_n\rangle_{\partial T}.\\
\end{split}
\end{equation*}
 
The primal-dual weak Galerkin finite element scheme based on the weak formulation \eqref{weakform} for the Poisson problem \eqref{model} is described as follows.
\begin{algorithm}[PDWG Scheme]
  Find $(u_h;\lambda_h)\in M_h
\times W_{h}^0$ satisfying
\begin{eqnarray}\label{32}
s(\lambda_h, \sigma)+b(u_h,\sigma)&=&(f,\sigma_0) +\langle g, \sigma_n\rangle_{\Gamma},\quad  \forall \sigma\in W^0_h,\\
b(v, \lambda_h)&=& 0,  \qquad  \qquad  \qquad  \qquad \forall v\in M_{h}.\label{2}
\end{eqnarray}
\end{algorithm}

On each element $T$, denote by $Q_0$ the $L^2$ projection operator onto $P_k(T)$. For any $w\in H^2(\Omega)$, denote by $Q_h w$ the $L^2$ projection
onto the weak finite element space $W_h$ such that on each element
$T$,
$$
Q_hw=\{Q_0w,Q_bw, Q_b(\nabla w \cdot \bn)\}.
$$
Denote by $\mathcal{Q}^{k-1}_h$ the $L^2$ projection operator onto the space $M_h$.
 
\begin{lemma}\label{Lemma5.1}   The $L^2$ projection operators $Q_h$ and $\mathcal{Q}^{k-1}_h$ satisfy the following commuting property:
 \begin{equation}\label{div}
\Delta_{w}(Q_h w) = \mathcal{Q}_h^{k-1}( \Delta w),    \qquad w\in H^2(T).
\end{equation}
\end{lemma}
\begin{proof}
For any $\phi\in P_{k-1}(T)$, we have from \eqref{disvergence} and the usual integration by parts that 
\begin{eqnarray*} 
 &&(\Delta_{w}(Q_h w), \phi)_T\\
 &=& 
 (Q_0 w,\Delta \phi)_T-\langle Q_b w,  \nabla \phi\cdot
\bm{n}\rangle_{\partial T}+ \langle Q_b(\nabla w\cdot \bn),\phi\rangle_{\partial T} \\
 &=& 
 (w,\Delta \phi)_T-\langle w,  \nabla \phi\cdot
\bm{n}\rangle_{\partial T}+ \langle  \nabla w\cdot \bn,\phi\rangle_{\partial T}\\
 &=&  (\Delta w, \phi)_T 
   =  ( \mathcal{Q}_h^{k-1} \Delta w, \phi)_T.
\end{eqnarray*}
This completes the proof of the lemma. 
\end{proof}
\section{Existence and Uniqueness}
In this section, we shall establish the solution existence and uniqueness of the PDWG scheme (\ref{32})-(\ref{2}). 

In the weak finite element space $W_h$, we introduce a semi-norm induced from the stabilizer; i.e., 
    \begin{equation*}\label{EQ:triplebarnorm}
   \3bar \sigma \3bar = s(\sigma, \sigma)^{\frac{1}{2}}, \qquad \forall  \sigma\in W_h.
   \end{equation*}

  \begin{lemma} \label{lem3-new} 
    ({\it inf-sup} condition)  For any $v\in M_h$, there exists $\rho_v\in W_h^0$ satisfying
\begin{eqnarray} \label{EQ:inf-sup-condition-01}
   b(v, \rho_v) = \| v\|^2, \quad \3bar \rho_v \3bar\leq C\| v\|.
   \end{eqnarray}
   \end{lemma}
   \begin{proof}
  Consider an auxiliary problem that seeks $w$ such that 
\begin{equation}\label{dual_2}
    \begin{split}
         \Delta w=& v,\qquad \text{in}\ \Omega,\\
       w=&0, \qquad \text{on}\ \Gamma.
    \end{split}
\end{equation}
We assume that the problem \eqref{dual_2} has the $H^{2}$-regularity; i.e., there exists a constant $C$ independent of $w$ satisfying
\begin{equation}\label{dual:regularity_2}
\|w\|_{2} \leq C \|v\|.
\end{equation}

We claim that $\rho_v=Q_h w=\{Q_0w, Q_bw, Q_b(\nabla w\cdot\bn)\}$ satisfies \eqref{EQ:inf-sup-condition-01}. Using the commutative property \eqref{div} and \eqref{dual_2} gives
  \begin{equation*}
      \begin{split}
          b(v, \rho_v) &=\sum_{T\in {\cal T}_h} (v, \Delta_w \rho_v)_T
          \\&=\sum_{T\in {\cal T}_h} (v, \Delta_w Q_h w)_T\\
          &=\sum_{T\in {\cal T}_h} (v, {\cal Q}_h^{k-1}\Delta w)_T\\
           &=\sum_{T\in {\cal T}_h} (v,  \Delta w)_T \\
            &=\sum_{T\in {\cal T}_h} (v, v)_T \\
            &= \|v\|^2.
      \end{split}
  \end{equation*}
  
  From the trace inequality (\ref{tracein}), and $H^{2}$-regularity \eqref{dual:regularity_2}, we have
   \begin{equation}\label{EQ:Estimate:002}
   \begin{split}
   &\sum_{T\in \mathcal{T}_h }h_T^{-3}\|Q_b\rho_0-\rho_b\|_{\pT}^2\\
  = &\sum_{T\in \mathcal{T}_h }h_T^{-3}\|Q_b Q_0w-Q_bw\|_{\pT}^2\\
     \leq&\sum_{T\in \mathcal{T}_h}
    h_T^{-3}\|Q_0w- w \|_{\pT}^2\\
  \leq & C\sum_{T\in \mathcal{T}_h}
    h_T^{-4}\|Q_0w- w \|_{T}^2+  h_T^{-2}\|Q_0w- w \|_{1,T}^2\\ 
        \leq &   C \|w\|_2^2 \\
  \leq & C \|v\|^2.
   \end{split}
   \end{equation}   
 Analogously, using the trace inequality \eqref{tracein}, and $H^{2}$-regularity \eqref{dual:regularity_2} gives
   \begin{equation}\label{EQ:Estimate:003}
    \begin{split}
    &\sum_{T\in \mathcal{T}_h }h_T^{-1}\|\nabla
    \rho_0\cdot\bn-\rho_n\|_{\pT}^2\\
   = & \sum_{T\in \mathcal{T}_h }h_T^{-1}\|\nabla
    Q_0 w \cdot\bn-Q_b(\nabla w\cdot\bn)\|_{\pT}^2\\
    \leq & \sum_{T\in \mathcal{T}_h }h_T^{-1}\|\nabla
    Q_0 w \cdot\bn- \nabla w\cdot\bn\|_{\pT}^2\\
        \leq &C \sum_{T\in \mathcal{T}_h }h_T^{-2} \|\nabla
    (Q_0-I) w \cdot\bn \|_{T}^2+ \|\nabla
    (Q_0-I) w \cdot\bn \|_{1,T}^2\\
         \leq &   C  \|w\|_2^2 \\
  \leq & C  \|v\|^2.
    \end{split}
   \end{equation}
   Combining the estimates (\ref{EQ:Estimate:002})-(\ref{EQ:Estimate:003}) yields
    \begin{equation*}\label{EQ:inf-sup-condition-02}
   \3bar \rho_v\3bar \leq C \| v\|.
    \end{equation*}   
    
    This completes the proof of the lemma.
\end{proof} 
%


\begin{theorem}\label{thmunique1}
The PDWG finite element algorithm (\ref{32})-(\ref{2}) has one and only one solution.
\end{theorem}
\begin{proof} It suffices to show that zero is the unique solution to the PDWG scheme (\ref{32})-(\ref{2}) with homogeneous data $f=0$  and $g=0$. To this end, assume $f=0$ and $g=0$ in (\ref{32})-(\ref{2}). By letting $v=u_h$ and $\sigma=\lambda_h$, the difference of (\ref{2}) and (\ref{32}) gives $s(\lambda_h,\lambda_h)=0$, which implies $Q_b\lambda_0=\lambda_b$ and $\nabla \lambda_0 \cdot \bn=\lambda_n$ on each $\partial T$. 

It follows from (\ref{2}), \eqref{disvergence} and the usual integration by parts that for any $v\in M_h$
\begin{equation}
\begin{split}
0=&b(v, \lambda_h)\\
=&\sum_{T\in \mathcal{T}_h}(v, \Delta_w \lambda_h)_T \\
=&\sum_{T\in \mathcal{T}_h}  (\Delta\lambda_0, v)_T+\langle Q_b\lambda_0- \lambda_b, \nabla v\cdot
\bm{n}\rangle_{\partial T}- \langle  \nabla \lambda_0\cdot \bn-\lambda_n, v \rangle_{\partial T}
\\
 =&\sum_{T\in \mathcal{T}_h}  (\Delta\lambda_0, v)_T,\\
\end{split}
\end{equation}
where we have used $Q_b\lambda_0=\lambda_b$ and $\nabla \lambda_0 \cdot \bn=\lambda_n$ on each $\partial T$. This implies $\Delta\lambda_0=0$ on each element $T\in\mathcal{T}_h$ by taking $v= \Delta\lambda_0$. 

We have from the fact $\Delta\lambda_0=0$ on each element $T\in\mathcal{T}_h$  and the usual integration by parts that
\begin{equation*}
    \begin{split}
      0= &\sum_{T\in {\cal T}_h}\int_T \Delta\lambda_0 \lambda_0 dT\\
       =&   \sum_{T\in {\cal T}_h}-\int_T \nabla\lambda_0 \cdot\nabla\lambda_0dT+\int_{\partial T} \nabla\lambda_0\cdot\bn \lambda_0 ds\\
       =&  \sum_{T\in {\cal T}_h}-\int_T  \nabla\lambda_0 \cdot\nabla\lambda_0dT+\int_{\partial T}  \nabla\lambda_0\cdot\bn Q_b\lambda_0 ds\\
         =& \sum_{T\in {\cal T}_h}-\int_T   \nabla\lambda_0 \cdot\nabla\lambda_0dT+\int_{\partial T}  \nabla\lambda_0\cdot\bn  \lambda_b ds\\
   =&-  \sum_{T\in {\cal T}_h}\int_T   \nabla\lambda_0 \cdot\nabla\lambda_0dT+\int_{\partial T} \lambda_n    \lambda_b ds\\ =&-  \sum_{T\in {\cal T}_h}\int_T   \nabla\lambda_0 \cdot\nabla\lambda_0dT,     
    \end{split}
\end{equation*}
where we used $Q_b\lambda_0=\lambda_b$ and $\lambda_n=\nabla\lambda_0\cdot\bn$ on each $\partial T$, and $\lambda_b=0$ on $\partial\Omega$. This leads to $\nabla \lambda_0=0$ on each $\partial T$. Therefore, we have $\lambda_0=C$ on each $T$. Using $Q_b\lambda_0=\lambda_b$  on each $\partial T$ we have $\lambda_b=C$ on each $\partial T$. Using  $\lambda_b=0$ on $\partial\Omega$, we have   $\lambda_b=0$ in $\Omega$ and further $\lambda_0=0$ in $\Omega$. Using $ \nabla \lambda_0 \cdot \bn=\lambda_n$ on each $\partial T$, we obtain $\lambda_n=0$  in $\Omega$. Therefore, we obtain $\lambda_h\equiv 0$ in $\Omega$.

Next we shall demonstrate $u_h\equiv 0$ in $\Omega$.   Using $\lambda_h\equiv 0$ in $\Omega$ and the equation (\ref{32}) gives
$$
b(u_h, \sigma)=0. 
$$
Using the inf-sup condition  \eqref{EQ:inf-sup-condition-01}, there exists a $\sigma\in W_h^0$ such that 
$$
0=b(u_h, \sigma)= \|u_h\|^2,
$$
which yields $u_h\equiv 0$ in $\Omega$. 

This completes the proof of the theorem.
\end{proof}

\section{Error Equations}\label{Section:error-equations} This section is devoted to deriving the error equations for the PDWG scheme (\ref{32})-(\ref{2}) which will play a critical role in establishing the error estimates in the following section.

Let $u$ and $(u_h, \lambda_h) \in  M_{h}\times W_{h}^0$ be the exact solution of the Poisson problem (\ref{model}) and its numerical approximation arising from the PDWG scheme (\ref{32})-(\ref{2}). Note that the Lagrange multiplier $\lambda_h$ approximates the trivial solution $\lambda=0$. 
The error functions are defined as the difference between the numerical solution $(u_h, \lambda_h)$ and the $L^2$ projection of the exact solution $u$ of (\ref{model}); i.e.,
\begin{align}\label{error}
e_h&=u_h-\mathcal{Q}^{k-1}_hu,\\
 \varepsilon_h&=\lambda_h-Q_h\lambda=\lambda_h.\label{error-2}
\end{align}

\begin{lemma}\label{Lemma:LocalEQ} 
For any $\sigma\in W_{h}$ and $v\in M_{h}$, the following identity holds true:
\begin{equation}\label{June30:008}
(\Delta_w \sigma, v)_T = (\Delta\sigma_0, v)_T +R_T(\sigma, v),
\end{equation}
where
\begin{equation}\label{EQ:April:06:100}
\begin{split}
R_T(\sigma,v) = \langle \sigma_0-\sigma_b, \nabla  v\cdot
\bm{n}\rangle_{\partial T}-\langle  \nabla \sigma_0\cdot \bn-\sigma_n,    v \rangle_{\partial T}.
\end{split}
\end{equation}
\end{lemma}

\begin{proof}
This proof can be easily obtained by using (\ref{disvergence}) and the usual integration by parts.
\end{proof}

\begin{lemma}\label{errorequa}
Let $u$ and $(u_h;\lambda_h) \in M_{h}\times W_{h}^0$ be the solutions arising from (\ref{model}) and (\ref{32})-(\ref{2}), respectively. 
The error functions $e_h$ and $\varepsilon_h$ satisfy the following error equations:
\begin{eqnarray}\label{sehv}
s(\varepsilon_h, \sigma)+b(e_h, \sigma)&=&\ell_u(\sigma),\qquad \forall\  \sigma\in W_{h}^0,\\
b(v, \varepsilon_h)&=&0,\qquad\qquad \forall v\in M_{h}, \label{sehv2}
\end{eqnarray}
where $\ell_u(\sigma)$ is given by
\begin{equation}\label{lu}
\begin{split}
\qquad \ell_u(\sigma) =&\sum_{T\in \mathcal{T}_h} \langle {\color{red} \sigma_0}-\sigma_b,  \nabla (u-\mathcal{Q}_h^{k-1}u)\cdot \bm{n}\rangle_{\partial T}\\&-\langle \nabla \sigma_0\cdot \bn-\sigma_n, u-\mathcal{Q}_h^{k-1}u\rangle_{\partial T}.
\end{split}
\end{equation}
\end{lemma}

\begin{proof} From (\ref{error-2}) and (\ref{2}) we have
\begin{align*}
b(v, \varepsilon_h) = b(v, \lambda_h) = 0,\qquad \forall v\in M_{h},
\end{align*}
which gives rise to (\ref{sehv2}).

Recall that $\lambda=0$. From (\ref{32}) we arrive at
\begin{equation}\label{EQ:April:04:100}
\begin{split}
 & s(\lambda_h-Q_h\lambda, \sigma)+b(u_h-\mathcal{Q}^{k-1}_hu, \sigma) \\
   = &(f, \sigma_0)+\langle g, \sigma_n\rangle_{\Gamma}-b(\mathcal{Q}^{k-1}_hu, \sigma).
\end{split}
\end{equation}
As to the term $b(\mathcal{Q}^{k-1}_hu, \sigma)$, using Lemma  \ref{Lemma:LocalEQ} and the usual integration by parts gives
\begin{equation}\label{EQ:April:04:103}
\begin{split}
&b(\mathcal{Q}^{k-1}_hu, \sigma) \\
= & \sum_{T\in \mathcal{T}_h} (\mathcal{Q}^{k-1}_hu,  \Delta_w \sigma)_T \\
=& \sum_{T\in \mathcal{T}_h} (\Delta\sigma_0, \mathcal{Q}_h^{k-1}u)_T +R_T(\sigma, \mathcal{Q}_h^{k-1}u)\\
= & \sum_{T\in \mathcal{T}_h} (\Delta\sigma_0, u)_T   +R_T(\sigma, \mathcal{Q}_h^{k-1}u)\\
=&\sum_{T\in \mathcal{T}_h}  (\sigma_0,  \Delta u)_T-\langle {\color{red}{\sigma_0}}, \nabla u \cdot\bn \rangle_{\partial T}+\langle \nabla \sigma_0\cdot \bn, u\rangle_{\partial T} +R_T(\sigma, \mathcal{Q}_h^{k-1}u)\\
=& (\sigma_0, f) +\sum_{T\in \mathcal{T}_h} -\langle \sigma_0-\sigma_b, \nabla u \cdot\bn \rangle_{\partial T}+\langle \nabla \sigma_0\cdot \bn-\sigma_n, u\rangle_{\partial T} \\& +R_T(\sigma, \mathcal{Q}_h^{k-1}u)+\langle \sigma_n, g\rangle_{\Gamma},\\
\end{split}
\end{equation}
where we used  $\sum_{T\in \mathcal{T}_h}  \langle \sigma_b, \nabla u \cdot \bn \rangle_{\partial T} = 
\langle \sigma_b, \nabla u \cdot \bn \rangle_{\Gamma}
$,  $\sigma_b=0$ on $\Gamma$, (\ref{model}) and
$\sum_{T\in \mathcal{T}_h}  \langle \sigma_n, u\rangle_{\partial T} = 
\langle \sigma_n, g\rangle_{\Gamma}
$.

Substituting (\ref{EQ:April:04:103})  into (\ref{EQ:April:04:100}) gives rise to the error equation (\ref{sehv}), which completes the proof of the lemma.
\end{proof}

A weak function $\sigma\in W_h^0$ is said to be {\em weakly harmonic} if $\Delta_w \sigma =0$. The error equation \eqref{sehv2} asserts that the error function $\varepsilon_h$ is weakly harmonic.

\begin{lemma}\label{weakly-harmonic}
The following estimate holds true for any weak finite element functions $\sigma\in W_h^0$:
\begin{eqnarray}\label{June30:001}
\sum_{T\in \T_h} \|\nabla \sigma_0\|_T^2 \leq C_1 \|\Delta_w \sigma\|_0^2 + C_2 h^2 s(\sigma,\sigma). 
\end{eqnarray}
In particular, for weakly harmonic finite element functions, we have
\begin{eqnarray}\label{June30:002}
\sum_{T\in \T_h} \|\nabla \sigma_0\|_T^2 \leq C h^2 s(\sigma,\sigma). 
\end{eqnarray}
\end{lemma}

\begin{proof}
From \eqref{June30:008} we have
$$
(\Delta_w \sigma, v)_T = (\Delta\sigma_0, v)_T +R_T(\sigma, v),\qquad \forall v\in M_h.
$$
By letting $v=-\mathcal{Q}_h^{k-1} \sigma_0$ we arrive at
\begin{eqnarray}
-(\Delta_w \sigma, \mathcal{Q}_h^{k-1} \sigma_0)_T &=& -(\Delta\sigma_0, \mathcal{Q}_h^{k-1} \sigma_0)_T - R_T(\sigma, \mathcal{Q}_h^{k-1} \sigma_0)\nonumber \\
&=& -(\Delta\sigma_0, \sigma_0)_T - R_T(\sigma, \mathcal{Q}_h^{k-1} \sigma_0)\label{June30:010}\\
&=& (\nabla\sigma_0, \nabla\sigma_0)_T-\langle\nabla\sigma_0\cdot\bn, \sigma_0\rangle_\pT - R_T(\sigma, \mathcal{Q}_h^{k-1} \sigma_0).\nonumber
\end{eqnarray}
Note that
\begin{eqnarray*}
\sum_{T\in\T_h} \langle\nabla\sigma_0\cdot\bn, \sigma_0\rangle_\pT = \sum_{T\in\T_h} \langle\nabla\sigma_0\cdot\bn - \sigma_n, \sigma_0\rangle_\pT + \langle \sigma_n, \sigma_0-\sigma_b\rangle_\pT.
\end{eqnarray*}
Thus,
\begin{eqnarray*}
&&\sum_{T\in\T_h} \langle\nabla\sigma_0\cdot\bn, \sigma_0\rangle_\pT + R_T(\sigma, \mathcal{Q}_h^{k-1} \sigma_0)\\
&=& \sum_{T\in\T_h} \langle\nabla\sigma_0\cdot\bn - \sigma_n, \sigma_0\rangle_\pT + \langle \sigma_n, \sigma_0-\sigma_b\rangle_\pT\\
&& +\langle \sigma_0-\sigma_b, \nabla \mathcal{Q}_h^{k-1} \sigma_0\cdot
\bm{n}\rangle_{\partial T}-\langle  \nabla \sigma_0\cdot \bn-\sigma_n, \mathcal{Q}_h^{k-1} \sigma_0 \rangle_{\partial T}\\
&=& \sum_{T\in\T_h} \langle\nabla\sigma_0\cdot\bn - \sigma_n, \sigma_0-\mathcal{Q}_h^{k-1} \sigma_0\rangle_\pT \\
&& + \sum_{T\in\T_h} \langle \sigma_n -\nabla\sigma_0\cdot\bn, \sigma_0-\sigma_b\rangle_\pT\\
&& +\sum_{T\in\T_h}\langle \sigma_0-\sigma_b, \nabla\sigma_0\cdot
\bm{n} + \nabla \mathcal{Q}_h^{k-1} \sigma_0\cdot
\bm{n}\rangle_{\partial T}\\
&=& I_1 + I_2 + I_3,
\end{eqnarray*}
where each $I_j$ stands for the corresponding term with $j=1,2,3$. Substituting the above identity into \eqref{June30:010} yields
\begin{equation}\label{June30:011}
    \sum_{T\in\T_h}(\nabla\sigma_0, \nabla\sigma_0)_T = I_1+I_2+I_3 - \sum_{T\in\T_h}(\Delta_w \sigma, \mathcal{Q}_h^{k-1} \sigma_0)_T.
\end{equation}

The term $I_1$ can be estimated by using the Cauchy-Schwarz inequality and the trace inequality as follows:
\begin{eqnarray}\label{June30:012}
    |I_1|&=& \left| \sum_{T\in\T_h} \langle\nabla\sigma_0\cdot\bn - \sigma_n, \sigma_0-\mathcal{Q}_h^{k-1} \sigma_0\rangle_\pT\right|\\
    &\le& C \left(\sum_{T\in\T_h} h_T \|\nabla\sigma_0\cdot\bn - \sigma_n\|_\pT^2\right)^{\frac12} \|\nabla \sigma_0\|_0.\nonumber
\end{eqnarray}
The term $I_2$ can be handled by using the Cauchy-Schwarz inequality:
\begin{eqnarray}
    |I_2|&=& 
\left|\sum_{T\in\T_h} \langle \sigma_n -\nabla\sigma_0\cdot\bn, \sigma_0-\sigma_b\rangle_\pT\right|\nonumber\\
&=&\left|\sum_{T\in\T_h} \langle \sigma_n -\nabla\sigma_0\cdot\bn, Q_b\sigma_0-\sigma_b\rangle_\pT\right|\label{June30:014}\\
&\le &\left(\sum_{T\in\T_h} h_T \|\nabla\sigma_0\cdot\bn - \sigma_n\|_\pT^2\right)^{\frac12} \left(\sum_{T\in\T_h} h_T^{-1} \|Q_b\sigma_0 - \sigma_b\|_\pT^2\right)^{\frac12}.\nonumber
\end{eqnarray}
The term $I_3$ can be estimted in a similar fashion as follows:
\begin{eqnarray}\nonumber
    |I_3|&=& \left| \sum_{T\in\T_h}\langle \sigma_0-\sigma_b, \nabla\sigma_0\cdot
\bm{n} + \nabla \mathcal{Q}_h^{k-1} \sigma_0\cdot
\bm{n}\rangle_{\partial T}\right|\\
&=& \left| \sum_{T\in\T_h}\langle Q_b\sigma_0-\sigma_b, \nabla\sigma_0\cdot
\bm{n} + \nabla \mathcal{Q}_h^{k-1} \sigma_0\cdot
\bm{n}\rangle_{\partial T}\right|\label{June30:015}\\
&\le& C \left(\sum_{T\in\T_h} h_T^{-1} \|Q_b\sigma_0 - \sigma_b\|_\pT^2\right)^{\frac12} \|\nabla \sigma_0\|_0.\nonumber
\end{eqnarray}
The last term on the right-hand side of \eqref{June30:011} can be estimated by using the discrete Poincare inequality:
\begin{eqnarray}\label{June30:016}
&&\left| \sum_{T\in\T_h}(\Delta_w \sigma, \mathcal{Q}_h^{k-1} \sigma_0)_T\right| \\
&\leq& \| \Delta_w \sigma\| \ \| \mathcal{Q}_h^{k-1} \sigma_0\| \nonumber\\
&\leq& \| \Delta_w \sigma\| \ \| \sigma_0\| \nonumber\\
&\le & C \| \Delta_w \sigma\| \left( \|\nabla\sigma_0\| + \left(\sum_{T\in\T_h} h_T^{-1} \|Q_b\sigma_0 - \sigma_b\|_\pT^2\right)^{\frac12}\right)\nonumber
\end{eqnarray}
Finally, combining \eqref{June30:011} with \eqref{June30:012}-\eqref{June30:016} yields the desire inequality \eqref{June30:001}.
\end{proof}

\section{Error Estimates}\label{Section:Stability}
In this section, we shall demonstrate the optimal order of error estimates for the PDWG scheme (\ref{32})-(\ref{2}).



\begin{theorem} \label{theoestimate} 
 Let $k\ge 1$. Let $u$ be the exact solution of the Poisson equation (\ref{model}) and $(u_h, \lambda_h) \in M_{h} \times W_{h}^0$ be the numerical solution arising from PDWG method (\ref{32})-(\ref{2}). Assume the exact solution $u$ is sufficiently regular such that $u\in H^{k}(\Omega)$. The following error estimate holds true; i.e.,
 \begin{equation}\label{erres}
\3bar \varepsilon_h \3bar+\|e_h\|
 \leq Ch^{k}\|u\|_{k}.
\end{equation}
\end{theorem}

\begin{proof}
Letting $\sigma=\varepsilon_h=\{\varepsilon_0,\varepsilon_b,\varepsilon_n\}$ in
the error equation (\ref{sehv}) and using (\ref{sehv2}) and \eqref{lu} we arrive at
\begin{equation}\label{EQ:April7:001}
\begin{split}
&s(\varepsilon_h, \varepsilon_h) = \ell_u(\varepsilon_h)\\
   =&\sum_{T\in \mathcal{T}_h} \langle \varepsilon_0-\varepsilon_b, \nabla ( u-\mathcal{Q}_h^{k-1}u) \cdot
\bm{n}\rangle_{\partial T}  +\langle \varepsilon_n -  \nabla \varepsilon_0\cdot \bn, u-\mathcal{Q}_h^{k-1}u\rangle_{\partial T}\\
=&\sum_{T\in \mathcal{T}_h} \langle Q_b\varepsilon_0-\varepsilon_b, \nabla ( u-\mathcal{Q}_h^{k-1}u) \cdot
\bm{n}\rangle_{\partial T}  +\langle \varepsilon_n -  \nabla \varepsilon_0\cdot \bn, u-\mathcal{Q}_h^{k-1}u\rangle_{\partial T}\\
& + \sum_{T\in \mathcal{T}_h} \langle (I-Q_b)\varepsilon_0, \nabla ( u-\mathcal{Q}_h^{k-1}u) \cdot
\bm{n}\rangle_{\partial T}.
\end{split}
\end{equation}

Using Cauchy-Schwarz inequality, the trace inequality \eqref{tracein} gives
\begin{equation}\label{EQ:J2}
\begin{split} 
 & \left|\sum_{T\in \mathcal{T}_h} \langle Q_b\varepsilon_0-\varepsilon_b, \nabla ( u-\mathcal{Q}_h^{k-1}u) \cdot
\bm{n}\rangle_{\partial T}\right|\\
 \leq & \Big(\sum_{T\in \mathcal{T}_h}h_T^{-3}\|Q_b\varepsilon_0-\varepsilon_b\|_\pT^2\Big)^{\frac{1}{2}} \Big(\sum_{T\in \mathcal{T}_h} h_T^3\| \nabla ( u-\mathcal{Q}_h^{k-1}u)\|_\pT^2\Big)^{\frac{1}{2}} \\
\leq & C\Big(\sum_{T\in \mathcal{T}_h} h_T^{2}\| u-\mathcal{Q}_h^{k-1}u\|^2_{1,T} + h_T^{4}\|u-\mathcal{Q}_h^{k-1}u\|_{2, T}^2\Big)^{\frac{1}{2}} \\
& \;\ \Big(\sum_{T\in \mathcal{T}_h}h_T^{-3}\|Q_b\varepsilon_0-\varepsilon_b\|_\pT^2\Big)^{\frac{1}{2}} \\
\leq &\ Ch^{k}\|u\|_{k
} s(\varepsilon_h, \varepsilon_h)^{\frac{1}{2}}.
\end{split}
\end{equation} 

Again, from the Cauchy-Schwarz inequality and the trace inequality \eqref{tracein} we obtain
\begin{equation}\label{EQ:J3}
\begin{split}
 &\left| \sum_{T\in \mathcal{T}_h}\langle \varepsilon_n - \nabla \varepsilon_0\cdot \bn, u-\mathcal{Q}_h^{k-1}u\rangle_{\partial T}\right|\\
\leq &\Big(\sum_{T\in \mathcal{T}_h} h_T^{-1}\|\varepsilon_n- \nabla \varepsilon_0\cdot \bn \|_\pT^2\Big)^{\frac{1}{2}} \Big(\sum_{T\in \mathcal{T}_h} h_T \|u-\mathcal{Q}_h^{k-1}u\|_\pT^2\Big)^{\frac{1}{2}}\\
\leq &C\Big(\sum_{T\in \mathcal{T}_h}h_T^{-1}\|\varepsilon_n- \nabla \varepsilon_0\cdot \bn \|_\pT^2\Big)^{\frac{1}{2}} \\
&\quad \cdot \Big(\sum_{T\in \mathcal{T}_h} \|u-\mathcal{Q}_h^{k-1}u\|_T^2 + h_T^2 \|\nabla(u-\mathcal{Q}_h^{k-1}u)\|_{T}^2 \Big)^{\frac{1}{2}}\\
\leq &\  Ch^{k}\|u\|_{k} s(\varepsilon_h, \varepsilon_h)^{\frac{1}{2}}.
\end{split}
\end{equation} 

Using the Cauchy-Schwarz inequality, the trace inequality \eqref{tracein}, and the estimate \eqref{June30:002} we arrive at

\begin{equation}\label{EQ:J3p5}
\begin{split}
&\left|\sum_{T\in \mathcal{T}_h} \langle (I-Q_b)\varepsilon_0, \nabla ( u-\mathcal{Q}_h^{k-1}u) \cdot
\bm{n}\rangle_{\partial T}\right|\\
=& \left|\sum_{T\in \mathcal{T}_h} \langle (I-Q_b)\varepsilon_0, (I-Q_b)\nabla u \cdot
\bm{n}\rangle_{\partial T}\right|\\
\leq & C\Big(\sum_{T\in \mathcal{T}_h} \| (I-\mathcal{Q}_h^{k-1})\nabla u\|^2_{T} + h_T^{2}\|(I-\mathcal{Q}_h^{k-1})\nabla u\|_{1, T}^2\Big)^{\frac{1}{2}} \\
& \ \cdot \Big(\sum_{T\in \mathcal{T}_h}h_T^{-1}\|(I-Q_b)\varepsilon_0\|_\pT^2\Big)^{\frac{1}{2}}\\
\le& Ch^{k-1}\|u\|_k \ h^{-1} \|(I-\mathcal{Q}_h^{k-1})\varepsilon_0\|\\
\le& C h^{k-1}\|u\|_k \|\nabla\varepsilon_0\|_0\\
\le& C h^k \|u\|_k s(\varepsilon_h,\varepsilon_h)^{\frac12}.
\end{split}
\end{equation}

Substituting the estimates \eqref{EQ:J2}-\eqref{EQ:J3p5} into  \eqref{EQ:April7:001} yields
\begin{equation}\label{aij}
s(\varepsilon_h, \varepsilon_h)  = |\ell_u(\varepsilon_h)|\leq Ch^{k}\|u\|_{k}s(\varepsilon_h, \varepsilon_h)^{\frac{1}{2}} ,
\end{equation} 
which leads to
\begin{equation}\label{EQ:April7:002}
s(\varepsilon_h, \varepsilon_h)^{\frac{1}{2}}  \leq C h^{k}\|u\|_{k}.
\end{equation}

Next, for the error function $e_h=u_h-{\cal Q}_h^{k-1}u$, from the inf-sup condition \eqref{EQ:inf-sup-condition-01}, there exists a $\rho\in W_h^0$ such that
\begin{equation}\label{inf}
    b(e_h, \rho)=\|e_h\|^2, \qquad \3bar \rho \3bar\leq C\|e_h\|.
\end{equation}
From the error equation \eqref{sehv} we have
$$
b(e_h, \rho)=\ell_u (\rho)-s(\varepsilon_h, \rho).
$$
Using the Cauchy-Schwarz inequality, the triangle inequality, \eqref{aij} and \eqref{EQ:April7:002}   that
\begin{equation*}
    \begin{split}
        |b(e_h, \rho)|\leq |\ell_u (\rho)| + \3bar\varepsilon_h\3bar \3bar \rho\3bar\leq Ch^{k}\|u\|_{k}\3bar \rho\3bar,
    \end{split}
\end{equation*}
which, combined with \eqref{inf}, gives
$$
\|e_h\|\leq Ch^{k}\|u\|_{k}. 
$$
The above equation, together with the error estimate \eqref{EQ:April7:002}, completes the proof of the theorem.
\end{proof}

 \section{Error Estimates for the Dual Variable}\label{Section:ES-2}
In this section  we shall establish some error estimates for the approximate dual variable $\lambda_h$
  in the $L^2$ norm. To this end, let $\varphi$ be the solution of the following auxiliary problem
  \begin{eqnarray}\label{dual-problem:new_1}
  \begin{aligned}
    \Delta \varphi   &=\theta, \ \ \mbox{ in} \ \Omega, \\
   \varphi&=0,\ \  \mbox{ on} \ \partial\Omega,
  \end{aligned}
  \end{eqnarray}
  where $\theta$ is a given function in $L^2(\Omega)$. Assume the dual problem (\ref{dual-problem:new_1}) has the $H^{2}$-regularity in the sense that there exists a constant $C$ such that
  \begin{equation}\label{regularity:1}
 \|\varphi\|_{2}\leq C \|\theta\| .
 \end{equation}
 From \eqref{disvergence} and the usual integration by parts we have for any $v\in W_h^0$
 \begin{equation*}
  \begin{split}
   & (\Delta_w v, \varphi)  \\
  = & \sum_{T\in\T_h} (\Delta_w v, {\cal Q}^{k-1}_h\varphi)_T\\ 
  = & \sum_{T\in\T_h} (\Delta v_0, {\cal Q}^{k-1}_h\varphi)_T+\langle Q_b v_0-v_b,   \nabla {\cal Q}^{k-1}_h\varphi\cdot
\bm{n}\rangle_{\partial T}+ \langle v_n- \nabla v_0\cdot\bn,{\cal Q}^{k-1}_h\varphi\rangle_{\partial T}\\
  = & \sum_{T\in\T_h} (v_0,\Delta\varphi)_T+\langle Q_bv_0-v_b, \nabla ( {\cal Q}^{k-1}_h\varphi-\varphi)\cdot
\bm{n}\rangle_{\partial T}+ \langle v_n- \nabla v_0\cdot\bn,{\cal Q}^{k-1}_h\varphi-\varphi\rangle_{\partial T},\\
  \end{split}
  \end{equation*}
where we used $\sum_{T\in {\cal T}_h} \langle v_b, \nabla\varphi \cdot\bn\rangle_{\partial T}= \langle v_b, \nabla \varphi\cdot\bn\rangle_{\partial \Omega}=0$ due to $v_b=0$ on $\partial\Omega$ and $\sum_{T\in {\cal T}_h} \langle v_n, \varphi\rangle_{\partial T}= \langle v_n,\varphi\rangle_{\partial \Omega}=0$ due to $\varphi=0$ on $\partial\Omega$. Thus, we have 
  \begin{equation}\label{e:11}
  \begin{split}
   &(v_0, \Delta\varphi)=   (\Delta_w v, \varphi)\\
   &- \sum_{T\in\T_h}\Big(\langle  Q_bv_0-v_b, \nabla ( {\cal Q}^{k-1}_h\varphi-\varphi)\cdot
\bm{n}\rangle_{\partial T}+ \langle v_n- \nabla v_0\cdot\bn,{\cal Q}^{k-1}_h\varphi-\varphi\rangle_{\partial T}\Big).
  \end{split}
  \end{equation}

  We have the following error estimates for the variable $\lambda_0$.

  \begin{theorem}\label{theo:1}
  Let $u$ and $(u_h;\lambda_h) \in M_h\times W_h^0$ be the  solutions of  \eqref{model} and (\ref{32})-(\ref{2}), respectively. Assume that the dual problem \eqref{dual-problem:new_1}
  has the $H^{2}(\Omega)$ regularity with the a priori estimate \eqref{regularity:1}.
  Then the following estimate holds true:
  \begin{equation}\label{es:2}
   \|\lambda_0\| \le C h^{k+2}\|u\|_{k}.
  \end{equation}
 \end{theorem}
  \begin{proof}
  For any given function $\theta\in L^2(\Omega)$, let $\varphi$ be the solution of \eqref{dual-problem:new_1}. From  \eqref{e:11}  we have
   \begin{equation*}
   \begin{split}
   (\theta, \lambda_0) = &(\lambda_0,\Delta\varphi)\\
   =& (\Delta_w \lambda_h, \varphi) \\
   & -\sum_{T\in\T_h}\Big(\langle Q_b\lambda_0-\lambda_b,   \nabla ( {\cal Q}^{k-1}_h\varphi-\varphi)\cdot
\bm{n}\rangle_{\partial T}+ \langle \lambda_n- \nabla \lambda_0\cdot\bn,{\cal Q}^{k-1}_h\varphi-\varphi\rangle_{\partial T}\Big)\\
     =&I_1+I_2+I_3.
  \end{split} 
  \end{equation*}
  We next estimate $I_i,\ 1\le i\le 3$,  respectively.   In light of \eqref{2}, we have $\Delta_w \lambda_h=0$ so that $I_1 = 0$. For the term $I_2$, we have from the Cauchy-Schwarz inequality and the trace inequality    \eqref{tracein}  
\begin{eqnarray*}
 |I_2|&=& | \sum_{T\in\T_h}\langle Q_b\lambda_0-\lambda_b, \nabla ( {\cal Q}^{k-1}_h\varphi-\varphi)\cdot
\bm{n}\rangle_{\partial T}|\\
 &\leq & \left(\sum_{T\in\T_h} h_T^{-3}  \| Q_b\lambda_0-\lambda_b\|^2_{\pT}\right)^{\frac12} \left( \sum_{T\in\T_h} h_T ^3 \|  \nabla ( {\cal Q}^{k-1}_h\varphi-\varphi)\cdot
\bm{n}\|^2_{\pT}\right)^{\frac12}\\
 &\leq & Cs(\lambda_h,\lambda_h)^{\frac 12}(\sum_{T\in\T_h}h_T ^2 \| \nabla ( {\cal Q}^{k-1}_h\varphi-\varphi)\cdot
\bm{n}\|^2_T+h_T^4 \|\nabla( \nabla ( {\cal Q}^{k-1}_h\varphi-\varphi)\cdot
\bm{n})\|^2_T)^{\frac12}
  \\&\leq & Ch^2 s(\lambda_h,\lambda_h)^{\frac 12}\|\varphi\|_{2}.
 \end{eqnarray*}
 As to $I_3$, we have from the Cauchy-Schwarz inequality and the trace inequality \eqref{tracein},
 \begin{eqnarray*}
 |I_3|&=& \left|  \sum_{T\in\T_h}\langle
   \varphi-{\cal Q}^{k-1}_h \varphi, \nabla \lambda_0\cdot{\bf n}-\lambda_n\rangle_\pT    \right|\\
 &\leq & \left(\sum_{T\in\T_h} h_T ^{-1} \| \nabla \lambda_0\cdot{\bf n}-\lambda_n\|^2_{\pT}\right)^{\frac12} \left( \sum_{T\in\T_h} h_T  \|\varphi-{\cal Q}^{k-1}_h \varphi \|^2_{\pT}\right)^{\frac12}\\
 &\leq & Cs(\lambda_h,\lambda_h)^{\frac 12}(\sum_{T\in\T_h} \|\varphi-{\cal Q}^{k-1}_h \varphi\|^2_T+ h_T^{2}\|\nabla(\varphi-{\cal Q}^{k-1}_h \varphi)\|^2_T)^{\frac12}\\
  &\leq & Ch^2  s(\lambda_h,\lambda_h)^{\frac 12}\|\varphi\|_{2}.
 \end{eqnarray*}

Combining all the estimates for $I_i$ and the estimate \eqref{erres} in Theorem \ref{theoestimate},  we arrive at
 \begin{equation}\label{eq:12}
  |(\lambda_0, \theta)| \leq Ch^2  s(\lambda_h,\lambda_h)^{\frac 12}\|\varphi\|_{2}\leq Ch^{k+2}\|u\|_{k}\|\varphi\|_{2}.
  \end{equation}
  The estimate \eqref{es:2} then follows from the $H^{2}(\Omega)$-regularity  \eqref{regularity:1}. This completes the proof.
  \end{proof}

  \section{Primal-Dual Weak Galerkin based on Domain Decompositions}
Let $\Omega$ be a bounded domain with a Lipschitz boundary $\partial\Omega$. Let $\{\Omega_j: j=1,\cdots,M\}$ be a domain decomposition of the whole $\Omega$ such that 
\begin{equation*}
\overline{\Omega}=\bigcup_{j=1}^M \overline{\Omega}_j, \qquad \Omega_j\cap\Omega_k=\emptyset, j\neq k.
\end{equation*}
Assume that $\Omega_j$ is star-shaped and $\partial\Omega_j$ ($j=1, \cdots, M$) is Lipschitz continuous. In practice, with the exception of a few $\Omega_j$ along $\partial\Omega$, each $\Omega_j$ is convex with a piecewise-smooth boundary. We introduce
$$
\Gamma=\partial \Omega, \qquad \Gamma_j=\Gamma\cap\partial\Omega_j,
\qquad \Gamma_{jk}=\Gamma_{kj}=\partial\Omega_j\cap\partial\Omega_k.
$$

For $j=1, \cdots, M$, patching $W_k(T)$ over all the element $T\in {\cal T}_h^i$ through a common value $v_b$ and $v_n$ on the interior edges or flat faces ${\mathcal
E}_h^0 \cap \Omega_j$ gives rise to $ W_h(\Omega_j)$.
We introduce two  subspaces of $W_h(\Omega_j)$ for $j=1, \cdots, M$; i.e.,
$$
 W_h^0(\Omega_j)=\{v\in  W_h(\Omega_j): v_b|_{\Gamma_j}=0 \}, \qquad j=1, \cdots, M,$$
 $$
 W_h^g(\Omega_j)=\{v\in  W_h(\Omega_j): v_b|_{\Gamma_j}=Q_b g \}, \qquad  j=1, \cdots, M.$$ 
We further introduce 
$${\cal W}_h=\Pi_{j=1}^M W_h(\Omega_j).$$
The two subspaces of ${\cal W}_h$   are defined as follows: 
$${\cal W}_h^g=\{v\in {\cal W}_h: v_b|_{e}=Q_bg,  e\subset\partial\Omega\},$$
$${\cal W}_h^0=\{v\in {\cal W}_h: v_b|_{e}=0,  e\subset\partial\Omega\}.$$

 Define the jump of $v\in {\cal W}_h$ on $\Gamma_{jk}$ for $j, k=1, \cdots, M$ by
  \begin{equation}\label{jump}
    \ljump v_b\rjump_{\Gamma_{jk}}=
    v_{b, jk} -v_{b, kj},
  \end{equation}
  where $ v_{b, jk}$ and $v_{b, kj}$ represent the values of $v_b$ on $\Gamma_{jk}$ as seen from $\Omega_{j}$ and $\Omega_{k}$ respectively.
  The order of $\Omega_j$ and $\Omega_k$ is non-essential in \eqref{jump} as long as the difference is taken in a consistent way in all the formulas. 

We further introduce the following spaces:
$${\cal V}_h=\{v\in {\cal W}_h: \ljump v_b\rjump_{\Gamma_{jk}}=0, \ljump v_n\rjump_{\Gamma_{jk}}=0, \ \forall j, k=1, \cdots, M\}.$$
The two subspaces of ${\cal V}_h$ are defined by
$$
{\cal V}_h^0=\{v\in {\cal V}_h: v_b|_{e}=0,  e\subset\partial\Omega\},
$$
$$
{\cal V}_h^g=\{v\in {\cal V}_h: v_b|_{e}=Q_bg,  e\subset\partial\Omega\}.
$$

We will introduce an equivalent form of the primal-dual weak Galerkin finite element method \eqref{32}-\eqref{2} which is  restricted in the subdomain $\Omega_j$ $(j=1, \cdots, M)$: Find $(u_{h, j}, \lambda_{h, j}, \mu_{h, j})\in M_h(\Omega_j)\times{\cal W}^0_h(\Omega_j)\times{\cal W}^g_h(\Omega_j)$ such that
\begin{equation}\label{pdwg}
\begin{split}
&s_{\Omega_j}(\lambda_{h, j}, w_j)+ \sum_{k=1}^M \langle \tau_{je}\mu_{n, jk}, w_{b, jk}\rangle_{\Gamma_{jk}}- \sum_{k=1}^M \langle \mu_{b, jk}, \tau_{je}w_{n, jk}\rangle_{\Gamma_{jk}} \\
&+(u_{h, j}, \Delta_w w_j)_{{\Omega_j}}=(f, w_{0, j})_{{\Omega_j}}+  \langle g, w_n \rangle_{\partial\Omega_j \cap \partial \Omega},\qquad\forall w_j\in {\cal W}^0_h(\Omega_j),\\
&(v_j, \Delta_w\lambda_{h, j})_{{\Omega_j}}=0,\qquad \forall v_j\in M_h(\Omega_j),\\
& \sum_{k=1}^M \langle \tau_{je}\nu_{n, jk}, \lambda_{b, jk}\rangle_{\Gamma_{jk}}- \sum_{k=1}^M \langle \nu_{b, jk}, \tau_{je}\lambda_{n, jk}\rangle_{\Gamma_{jk}} = 0,\quad \forall \nu_h, 
\end{split}
\end{equation} 
where $s_{\Omega_j}(\cdot, \cdot)=\sum_{T\in {\cal T}_h^i}s_T(\cdot, \cdot)$, $(\cdot, \cdot)_{\Omega_j}=\sum_{T\in {\cal T}_h^i} (\cdot, \cdot)_T$, and $\tau_{je}=\bn_j\cdot\bn_e$ ($\bn_j$ is an unit outward normal direction to $\Omega_j$. Note that $\bn_e$ is selected as an unit outward normal direction  on boundary edge or otherwise the term $\langle g, w_n \rangle_{\partial\Omega_j \cap \partial \Omega}$ in the numerical schemes has to be modified). 
In addition, the following interface conditions hold true:
 
 \begin{eqnarray}\label{c1}
 \lambda_{b, jk}&=&\lambda_{b, kj}, \qquad\lambda_{n, jk}=\lambda_{n, kj};\\
\label{c2}
\mu_{b, jk}&=&\mu_{b, kj}, \qquad\mu_{n, jk}=\mu_{n, kj}.
\end{eqnarray}  



\begin{lemma}
Let $\sigma$ and $\beta$ be two positive functions on $\bigcup_{j, k=1}^M \Gamma_{jk}$. The conditions \eqref{c1} and \eqref{c2} are equivalent to the following:

\begin{eqnarray}\label{a}
   \mu_{b, jk}+\sigma \lambda_{n, jk}\tau_{je}&=&\mu_{b, kj}-\sigma \lambda_{n, kj}\tau_{ke},\\ 
\label{b}    \beta \lambda_{b, jk}-\mu_{n, jk}\tau_{je} &=& \beta \lambda_{b, kj}+\mu_{n, kj}\tau_{ke},\\
\label{c}
   \mu_{b, kj}+\sigma \lambda_{n, kj}\tau_{ke}&=&\mu_{b, jk}-\sigma \lambda_{n, jk}\tau_{je},\\
\label{d}\ \beta \lambda_{b, kj}-\mu_{n, kj}\tau_{ke} &=& \beta \lambda_{b, jk}+\mu_{n, jk}\tau_{je}.
\end{eqnarray} 
 \end{lemma}

\begin{proof}
From \eqref{a} and \eqref{c}), we have
$$
\mu_{b, jk}=\mu_{b, kj}, \lambda_{n, jk}=\lambda_{n, kj}.
$$
From \eqref{b} and \eqref{d}, we have $\lambda_{b, jk}=\lambda_{b, kj}$  and  $\mu_{n, jk}=\mu_{n, kj}$. 

This completes the proof of the lemma.
\end{proof}
 
\section{Iterative Procedure for Primal-Dual Weak Galerkin Scheme}
The iterative procedure for the PDWG scheme \eqref{pdwg} defined on the subdomain $\Omega_j$ for $j=1, \cdots, M$ is as follows: Find $\lambda_{h, j}^{(m)}=\{\lambda_{0, j}^{(m)}, \lambda_{b, jk}^{(m)}, \lambda_{n, jk}^{(m)}\} \in {\cal W}^0_h(\Omega_j)$, $u_{h, j}^{(m)}\in  M_h(\Omega_j)$ such that
\begin{equation}\label{e1}
\left\{
\begin{split}
&s_{\Omega_j}(\lambda_{h, j}^{(m)}, w_j)+(u_{h, j}^{(m)}, \Delta_w w_j)_{{\Omega_j}}\\
& \ \ + \sum_{k=1}^M\langle
\beta\lambda_{b, jk}^{(m)}-r_{b,kj}^{(m-1)}, w_{b, jk}\rangle_{\Gamma_{jk}}+\sum_{k=1}^M\langle
\sigma\lambda_{n, jk}^{(m)}-r_{n,kj}^{(m-1)}, w_{n, jk}\rangle_{\Gamma_{jk}}\\
=& (f, w_{0,j})_{{\Omega_j}}+ \langle g, w_{n,j}\rangle_{\partial\Omega_j\cap\partial\Omega},\qquad\forall w_j\in {\cal W}^0_h(\Omega_j),\\
&(v_j, \Delta_w \lambda_{h, j}^{(m)})_{{\Omega_j}}=0,\qquad \forall  v_j \in M_h(\Omega_j),
\end{split}\right.
\end{equation}
where
\begin{eqnarray}
r_{b,kj}^{(m-1)} &=& 2\beta \lambda_{b,kj}^{(m-1)} - r_{b,jk}^{(m-2)},\label{EQ:March06:001}\\
 r_{n,kj}^{(m-1)} &=& 2\sigma \lambda_{n,kj}^{(m-1)} - r_{n,jk}^{(m-2)}.\label{EQ:March06:002}
\end{eqnarray}

\section{The derivation of the iterative scheme}

  First note the following:
\begin{equation}\label{c5-new}
     \mu_{b, jk}^{(m)}\tau_{je}+\sigma \lambda^{(m)}_{n, jk}=\tau_{je}\mu_{b, kj}^{(m-1)} + \sigma \lambda_{n, kj}^{(m-1)},
\end{equation}  
\begin{equation}\label{c6-new}
    \beta\lambda_{b, jk}^{(m)}-\mu_{n, jk}^{(m)}\tau_{je}=\beta \lambda_{b, kj}^{(m-1)}+\mu_{n, kj}^{(m-1)}\tau_{ke},
\end{equation}  
\begin{equation}\label{c5-newnew}
     \mu_{b, kj}^{(m-1)}\tau_{je}-\sigma \lambda^{(m-1)}_{n, kj}=\tau_{je}\mu_{b, jk}^{(m-2)} - \sigma \lambda_{n, jk}^{(m-2)}.
\end{equation} 
 
Substituting \eqref{c5-new}-\eqref{c5-newnew} into \eqref{pdwg} yields
\begin{equation*} 
\begin{split}
&\quad \ s_{\Omega_j}(\lambda_{h, j}^{(m)}, w_j) +(u_{h, j}^{(m)}, \Delta_w w_j)_{{\Omega_j}} 
\\ & \quad \ 
   +\sum_{k=1}^M\Big( \langle \beta \lambda_{b,jk}^{(m)}, w_{b, jk}\rangle_{\Gamma_{jk}}
   + \langle \sigma \lambda^{(m)}_{n, jk}, w_{n, jk}\rangle_{\Gamma_{jk}} \Big)
\\ &=
   \sum_{k=1}^M\Big(  
   \langle    \beta \lambda_{b, kj}^{(m-1)}+\mu_{n, kj}^{(m-1)}{ \tau_{ke}}, w_{b, jk}    \rangle_{\Gamma_{jk}} 
   + {  \langle  \sigma \lambda_{n, kj}^{(m-1)} + \mu_{b, kj}^{(m-1)} \tau_{je} , w_{n, jk}   \rangle_{\Gamma_{jk}} }
     \Big)
\\ & \quad \ 
   +  (f, w_{0,j})_{{\Omega_j}} + \langle g, w_{n,j}\rangle_{\partial\Omega_j\cap\partial\Omega}
\\ &=
   \sum_{k=1}^M\Big(  
   \langle    r_{b, kj}^{(m-1)}, w_{b, jk}    \rangle_{\Gamma_{jk}} 
   + \langle   r_{n, kj}^{(m-1)}, w_{n, jk}   \rangle_{\Gamma_{jk}}
     \Big) 
   +  (f, w_{0,j})_{{\Omega_j}} + \langle g, w_{n,j}\rangle_{\partial\Omega_j\cap\partial\Omega}
\\ &=
   \sum_{k=1}^M\Big(  
   \langle   2 \beta \lambda_{b, kj}^{(m-1)}
               -(\beta \lambda_{b, jk}^{(m-2)} + \mu_{n, jk}^{(m-2)}{ \tau_{je}}), w_{b, jk}    \rangle_{\Gamma_{jk}} 
   \\&\qquad 
  + \langle   {  2 \sigma \lambda_{n, kj}^{(m-1)} + (- \sigma \lambda_{n, jk}^{(m-2)} + \tau_{je} }\mu_{b, jk}^{(m-2)} ) , 
             w_{n, jk}   \rangle_{\Gamma_{jk}}
     \Big)
\\ & \quad \ 
   +  (f, w_{0,j})_{{\Omega_j}} + \langle g, w_{n,j}\rangle_{\partial\Omega_j\cap\partial\Omega} \\
&=
   \sum_{k=1}^M\Big(  
   \langle   2 \beta \lambda_{b, kj}^{(m-1)} -r_{b, jk}^{(m-2)}, w_{b, jk}    \rangle_{\Gamma_{jk}}  
  + \langle   {  2 \sigma \lambda_{n, kj}^{(m-1)} - r_{n, jk}^{(m-2)}}, 
             w_{n, jk}   \rangle_{\Gamma_{jk}}
     \Big)
\\ & \quad \ 
   +  (f, w_{0,j})_{{\Omega_j}}+ \langle g, w_{n,j}\rangle_{\partial\Omega_j\cap\partial\Omega}.
\end{split}
\end{equation*}

The iteration schemes is then given as follows:
\begin{equation} \label{iterativescheme}
\begin{split}
&\quad \ s_{\Omega_j}(\lambda_{h, j}^{(m)}, w_j) +(u_{h, j}^{(m)}, \Delta_w w_j)_{{\Omega_j}} 
\\ & \quad \ 
   +\sum_{k=1}^M\Big( \langle \beta \lambda_{b,jk}^{(m)} - r_{b,kj}^{(m-1)}, w_{b, jk}\rangle_{\Gamma_{jk}}
   + \langle \sigma \lambda^{(m)}_{n, jk}- r_{n,kj}^{(m-1)}, w_{n, jk}\rangle_{\Gamma_{jk}} \Big)
\\ &=  (f, w_{0,j})_{{\Omega_j}}+ \langle g, w_{n,j}\rangle_{\partial\Omega_j\cap\partial\Omega},
\end{split}
\end{equation}
where 
{ 
\begin{eqnarray} 
 r_{b,kj}^{(m-1)}&=& 2\beta  \lambda_{b, kj}^{(m-1)} -r_{b, jk}^{(m-2)},  r_{b,kj}^{(0)}=0, \label{cond:01}\\
                    r_{n, kj}^{(m-1)} &=& 2 \sigma \lambda_{n, kj}^{(m-1)} - r_{n, jk}^{(m-2)},   r_{n,kj}^{(0)}=0.\label{cond:02}  
\end{eqnarray} 
}

Observe tht the connection with the Lagrange multiplier $\mu_{b,jk}$ and $\mu_{n,jk}$ is given as follows:
\begin{eqnarray}
\mu_{n,jk}^{(m)} \tau_{je} &=& \beta\lambda_{b,jk}^{(m)} - r_{b,kj}^{(m-1)},\\
-\mu_{b,jk}^{(m)} \tau_{je} &=& \sigma\lambda_{n,jk}^{(m)} - r_{n,kj}^{(m-1)}.
\end{eqnarray}

{ 
The iteration schemes is given as follows:
\begin{equation} 
\begin{split}
&\quad \ s_{\Omega_j}(\lambda_{h, j}^{(m)}, w_j) +(u_{h, j}^{(m)}, \Delta_w w_j)_{{\Omega_j}} 
   +\sum_{k=1}^M\Big( \langle \beta \lambda_{b,jk}^{(m)}, w_{b, jk}\rangle_{\Gamma_{jk}}
   + \langle \sigma \lambda^{(m)}_{n, jk}, w_{n, jk}\rangle_{\Gamma_{jk}} \Big)
\\ &=  (f, w_{0,j})_{{\Omega_j}}+ \langle g, w_{n,j}\rangle_{\partial\Omega_j\cap\partial\Omega} 
   +\sum_{k=1}^M\Big( \langle {  r_{b,kj}^{(m-1)}}, w_{b, jk}\rangle_{\Gamma_{jk}}
   + \langle { r_{n,kj}^{(m-1)}}, w_{n, jk}\rangle_{\Gamma_{jk}} \Big) \\
&\qquad {  (v_j, \Delta_w \lambda_{h, j}^{(m)})_{{\Omega_j}}=0,\qquad \forall  v_j \in M_h(\Omega_j)},
\end{split}
\end{equation}
where  
\begin{eqnarray} 
 r_{b,kj}^{(m-1)}&=& 2{ \beta }\lambda_{b, kj}^{(m-1)} - r_{b, jk}^{(m-2)}, \qquad
      r_{b,kj}^{(0)}=0,\\
 r_{n, kj}^{(m-1)} &=&
    2 {  \sigma}\lambda_{n, kj}^{(m-1)} - r_{n, jk}^{(m-2)}, \qquad
    r_{n,kj}^{(0)}=0.   
\end{eqnarray} 
}

\section{Convergence Analysis}

From \eqref{EQ:March06:001} we have
\begin{eqnarray*}
\sum_{k,j} \|r_{b,kj}^{(m-1)}\|_{\Gamma_{kj}}^2 &=& \sum_{k,j} \|\beta\lambda_{b,kj}^{(m-1)} \|_{\Gamma_{kj}}^2   + \|r_{b,jk}^{(m-2)}- \beta\lambda_{b,kj}^{(m-1)}\|_{\Gamma_{kj}}^2\\
&& + 2\beta \langle \beta\lambda_{b,kj}^{(m-1)} - r_{b,jk}^{(m-2)}, \lambda_{b,kj}^{(m-1)}\rangle_{\Gamma_{kj}}\\
&=& \sum_{k,j} \|\beta\lambda_{b,kj}^{(m-1)} \|_{\Gamma_{kj}}^2   + \|r_{b,jk}^{(m-2)}- \beta\lambda_{b,kj}^{(m-1)}\|_{\Gamma_{kj}}^2\\
&& - 2\beta \langle \beta\lambda_{b,kj}^{(m-1)} - r_{b,jk}^{(m-2)}, \lambda_{b,kj}^{(m-1)}\rangle_{\Gamma_{kj}}\\
&& +4\beta \langle \beta\lambda_{b,kj}^{(m-1)} - r_{b,jk}^{(m-2)}, \lambda_{b,kj}^{(m-1)}\rangle_{\Gamma_{kj}}\\
&=& \sum_{k,j} \|r_{b,jk}^{(m-2)}\|_{\Gamma_{kj}}^2 + 4\beta \langle \beta\lambda_{b,kj}^{(m-1)} - r_{b,jk}^{(m-2)}, \lambda_{b,kj}^{(m-1)}\rangle_{\Gamma_{kj}}.
\end{eqnarray*}
Analogously, from \eqref{EQ:March06:002} we have
\begin{eqnarray*}
\sum_{k,j} \|r_{n,kj}^{(m-1)}\|_{\Gamma_{kj}}^2 &=& \sum_{k,j} \|\sigma\lambda_{n,kj}^{(m-1)} \|_{\Gamma_{kj}}^2   + \|r_{n,jk}^{(m-2)}- \sigma\lambda_{n,kj}^{(m-1)}\|_{\Gamma_{kj}}^2\\
&& + 2\sigma \langle \sigma\lambda_{n,kj}^{(m-1)} - r_{n,jk}^{(m-2)}, \lambda_{n,kj}^{(m-1)}\rangle_{\Gamma_{kj}}\\
&=& \sum_{k,j} \|\sigma\lambda_{n,kj}^{(m-1)} \|_{\Gamma_{kj}}^2   - \|r_{n,jk}^{(m-2)}- \sigma\lambda_{n,kj}^{(m-1)}\|_{\Gamma_{kj}}^2\\
&& - 2\sigma \langle \sigma\lambda_{n,kj}^{(m-1)} - r_{n,jk}^{(m-2)}, \lambda_{n,kj}^{(m-1)}\rangle_{\Gamma_{kj}}\\
&& +4\sigma \langle \sigma\lambda_{n,kj}^{(m-1)} - r_{n,jk}^{(m-2)}, \lambda_{n,kj}^{(m-1)}\rangle_{\Gamma_{kj}}\\
&=& \sum_{k,j} \|r_{n,jk}^{(m-2)}\|_{\Gamma_{kj}}^2 + 4\sigma \langle \sigma\lambda_{n,kj}^{(m-1)} - r_{n,jk}^{(m-2)}, \lambda_{n,kj}^{(m-1)}\rangle_{\Gamma_{kj}}.
\end{eqnarray*}
By letting $w_j=\lambda_{h,j}^{(m-1)}$ in \eqref{e1} at the iterative step $m-1$ we obtain
\begin{equation}\label{EQ:100}
\begin{split}
&\sum_{j,k=1}^M \langle \beta \lambda_{b,jk}^{(m-1)} - r_{b,kj}^{(m-2)}, \lambda_{b, jk}^{(m-1)}\rangle_{\Gamma_{jk}}
   + \langle \sigma \lambda^{(m-1)}_{n, jk}- r_{n,kj}^{(m-2)}, \lambda_{n, jk}^{(m-1)}\rangle_{\Gamma_{jk}} \\
= & -  \sum_j s_{\Omega_j}(\lambda_{h, j}^{(m-1)}, \lambda_{h, j}^{(m-1)}).
\end{split}
\end{equation}
It follows that
\begin{equation}\label{EQ:energy}
\begin{split}
&\beta^{-1}\sum_{k,j} \|r_{b,kj}^{(m-1)}\|_{\Gamma_{kj}}^2 + \sigma^{-1}\sum_{k,j} \|r_{n,kj}^{(m-1)}\|_{\Gamma_{kj}}^2 \\=&
\beta^{-1}\sum_{k,j} \|r_{b,kj}^{(m-2)}\|_{\Gamma_{kj}}^2 + \sigma^{-1}\sum_{k,j} \|r_{n,kj}^{(m-2)}\|_{\Gamma_{kj}}^2\\
& - 4 \sum_j s_{\Omega_j}(\lambda_{h, j}^{(m-1)}, \lambda_{h, j}^{(m-1)}).
\end{split}
\end{equation}

The energy identity \eqref{EQ:energy} has the following implications:

\begin{itemize}
\item The sequences $r_{b,jk}^{(m)} $ and $r_{n,jk}^{(m)}$ are bounded, and thus has convergent subsequences. May assume that this sequence itself is convergent.
\item $\sum_j s_{\Omega_j}(\lambda_{h, j}^{(m-1)}, \lambda_{h, j}^{(m-1)}) \to 0$ as $m\to \infty$.
\item From \eqref{EQ:March06:001} and \eqref{EQ:March06:002} we see that $\lambda_{b,jk}^{(m)} $ and $\lambda_{n,jk}^{(m)}$ are convergent.
\item From the second equation in \eqref{e1} and the fact that $\sum_j s_{\Omega_j}(\lambda_{h, j}^{(m-1)}, \lambda_{h, j}^{(m-1)}) \to 0$ we can show that $\Delta\lambda_0^{(m)} \to 0$. This leads to the result of $\lambda_0^{(m)} \to 0$ so that $\lambda_{b,jk}^{(m)} $ and $\lambda_{n,jk}^{(m)}$ are all convergent to zero.
\item From the first equation of \eqref{e1} we may show that $u_{h,j}^{(m)} \to 0 $ by special selections of $w_j$ such that $\Delta_w w_j =  u_{h,j}^{(m)}$.
\end{itemize}

\end{document}